\documentclass[11pt]{article}
\usepackage{amsthm}
\usepackage{amssymb}
\usepackage{amsmath,enumerate}
\usepackage{comment}
\usepackage{thm-restate}
\usepackage{url} 
\usepackage{hyperref}
\usepackage{graphicx}
\usepackage{subfigure}
\usepackage[noabbrev,capitalise]{cleveref}
\usepackage[affil-it]{authblk}
\usepackage{color}
\usepackage[normalem]{ulem}

\usepackage[margin=1in]{geometry}

\theoremstyle{plain}
\newtheorem{thm}{Theorem}[section]
\newtheorem{lem}[thm]{Lemma}

\newtheorem{conj}[thm]{Conjecture}

{\noindent \emph{Proof.} {}{#1}{}}{\hfill
	$\Diamond$\vspace{1em}}

\theoremstyle{plain} 
\newcommand{\thistheoremname}{}
\newtheorem{genericthm}[section]{\thistheoremname}

\theoremstyle{definition}

\def\es{\emptyset}
\def\less{\setminus}

\newcounter{counter}
 
\def\dfn#1{{\sl #1}}

\def\es{\emptyset}
\def\less{\setminus}

\newcommand{\ceil}[1]{{\left\lceil #1 \right\rceil}}

\title{Odd clique minors in graphs with independence number two}
\author{
{\small  Yuqing Ji$^a$, \ \ Zi-Xia Song$^{b,}$\thanks{Supported by  NSF grant DMS-2153945}
 }, \ \ Evan Weiss$^b$, \ \   Xia Zhang$^{a,}$\thanks{
Supported by the National Natural Science Foundation of China (No.12071265)}\\
 \footnotesize{ $^a$School of Mathematics and Statistics, Shandong Normal University, Jinan, 250358, China\\
 $^b$ Department of Mathematics, University of Central Florida, Orlando, FL 32816, USA}
}
 
\begin{document}
\maketitle

\begin{abstract}
A $K_t$-expansion consists of $t$ vertex-disjoint trees, every two of
which are joined by an edge. We call such an expansion odd if
its vertices can be two-colored so that the edges of the trees are
bichromatic but the edges between trees are monochromatic. A graph contains an odd $K_t$ minor or an odd clique minor of order $t$ if it contains an odd $K_t$-expansion.  Gerards and Seymour from 1995 conjectured that every graph $G$ contains an odd $K_{\chi(G)}$ minor, where $\chi(G)$ denotes the chromatic number of $G$.  This conjecture is referred to as ``Odd Hadwiger's Conjecture".   Let $\alpha(G)$ denote the independence number of a graph $G$.    In this paper we investigate the Odd Hadwiger's Conjecture for graphs $G$ with $\alpha(G)\le2$. We first observe that a graph $G$ on $n$ vertices with $\alpha(G)\le2$   contains an odd $K_{\chi(G)}$ minor if and only if $G$ contains an odd clique minor of order $\ceil{n/2}$. We then prove that  every graph $G$ on $n$ vertices with  $\alpha(G)\le 2$ contains an odd clique minor of order $\ceil{n/2}$  if $G$ contains a clique of order $n/4$ when $n$ is even and $(n+3)/4$ when $n$ is odd, or $G$ does not contain $H$ as an induced subgraph, where $\alpha(H)\le 2$ and $H$ is an induced subgraph of   
$K_1 + P_4$,  $K_2+(K_1\cup K_3)$,  $K_1+(K_1\cup K_4)$, $K_7^-$, $K_7$,  or the     kite graph. 

\end{abstract}

\section{Introduction}

All graphs in this paper are finite and simple. For a graph $G$ we use $|G|$, $e(G)$, $\delta (G)$, $\Delta(G)$, $\alpha(G)$, $\chi(G)$, $\omega(G)$ to denote the number
of vertices, number of edges,   minimum degree, maximum degree,  independence number,  chromatic number, and clique number   of $G$, respectively.  The \dfn{complement} of $G$ is denoted by $\overline{G}$.  
Given a graph $H$, we say that $G$ is \dfn{$H$-free} if $G$ has no induced subgraph isomorphic to $H$. For a family  $\mathcal{F}$ of graphs, we say that $G$ is $\mathcal{F}$-free if $G$ is $F$-free for every  $F\in \mathcal{F}$.   A \dfn{clique} in $G$ is a set of vertices all pairwise adjacent.   For a vertex $x\in V(G)$, we  use $N(x)$ to denote the set of vertices in $G$ which are adjacent to $x$.
We define $N[x] := N(x) \cup \{x\}$.  The degree of $x$ is denoted by $d_G(x)$ or
simply $d(x)$.   If  $A, B\subseteq V(G)$ are disjoint, we say that $A$ is \emph{complete} to $B$ if each vertex in $A$ is adjacent to all vertices in $B$, and $A$ is \emph{anticomplete} to $B$ if no vertex in $A$ is adjacent to any vertex in $B$.
If $A=\{a\}$, we simply say $a$ is complete to $B$ or $a$ is anticomplete to $B$.   We use $e_G(A, B)$ to denote the number of edges between $A$ and $B$ in  $G$. 
The subgraph of $G$ induced by $A$, denoted by $G[A]$, is the graph with vertex set $A$ and edge set $\{xy \in E(G) \mid x, y \in A\}$. We denote by $B \less A$ the set $B - A$,   and $G \less A$ the subgraph of $G$ induced on $V(G) \less A$, respectively. 
If $A = \{a\}$, we simply write $B \less a$    and $G \less a$, respectively.  
    The \dfn{join} $G+H$ (resp. \dfn{union} $G\cup H$) of two 
vertex-disjoint graphs
$G$ and $H$ is the graph having vertex set $V(G)\cup V(H)$  and edge set $E(G)
\cup E(H)\cup \{xy\, |\,  x\in V(G),  y\in V(H)\}$ (resp. $E(G)\cup E(H)$).  
We use the convention   ``A :="  to mean that $A$ is defined to be
the right-hand side of the relation. 
 For any positive integer $k$, we write  $[k]$ for the set $\{1, \ldots, k\}$ and    we use $K_k^-$ to denote the   graph  obtained from the complete graph $K_k$ by deleting one edge. \medskip

   Given graphs $G$ and $H$, we say that $G$ has \emph{an $H$ minor} if a graph isomorphic to $H$ can be obtained from a subgraph   of $G$ by contracting edges. A \dfn{$K_t$-expansion} consists of $t$ vertex-disjoint trees, every two of
which are joined by an edge. We call such an expansion \dfn{odd} if
its vertices can be two-colored so that the edges of the trees are
bichromatic but the edges between trees are monochromatic. A graph $G$ contains an \dfn{odd $K_t$ minor} or an \dfn{odd clique minor of order $t$} if it contains an odd $K_t$-expansion. It can be easily checked that every  graph that  has an 
odd $K_t$ minor
must  contain  $K_t$ as a minor. \medskip

Gerards and Seymour~\cite[Section 6.5]{jto} made the following  conjecture from 1995.

\begin{conj}\label{oddc}
For every integer $t\ge1$, every graph with no odd $K_t$ minor is $(t-1)$-colorable. \end{conj}

This conjecture is referred to as ``Odd Hadwiger's Conjecture". It is substantially stronger than Hadwiger's Conjecture which states that every graph with no $K_t$ minor is $(t-1)$-colorable. Conjecture \ref{oddc} is trivially true for $t \leq 3$. The case $t = 4$ was proved by Catlin \cite{cat}. When $t = 5$, Guenin \cite{gue} announced at a meeting in Oberwolfach a solution of the case $t = 5$. It remains open for $t \geq 6$. Geelen et al. \cite{ggr} proved that every graph with no odd $K_{t}$-minor is $O(t\sqrt{\log t})$-colorable. Subsequently, an asymptotical improvement of this upper bound to $O(t(\log t)^{\beta})$ for any $\beta > \frac{1}{4}$ was achieved by Norin and the second author~\cite{ns}. This was improved further to $O(t(\log \log t)^{6})$ by Postle~\cite{pos},  and later refined to  $O(t(\log \log t)^{2})$ by Delcourt and Postle \cite{dp}.  The current best known bound  is $O(t(\log \log t))$ due to 
Steiner~\cite{best}. For more information on  Odd Hadwiger's Conjecture, we refer the readers   to the surveys  \cite{ks, kaw, ste}.\\

 Let $oh(G)$ denote the largest integer $t$ such that $G$ has an odd $K_t$ minor. Then  Odd Hadwiger's Conjecture states that $oh(G)\ge \chi(G)$ for every graph $G$.  Note that  $|G|\le \alpha(G)\cdot \chi(G)$.    Odd Hadwiger's Conjecture implies the following weaker conjecture;  the odd-minor variant of Duchet-Meyniel Conjecture~\cite{dm} which states that  $h(G)\ge   \lceil |G|/\alpha(G) \rceil $ for every graph $G$, where   $h(G)$ denotes the largest $t$ such that $G$ contains $K_t$ as a minor. 

\begin{conj}\label{c:oddalpha2}
  $oh(G)\ge   \lceil |G|/\alpha(G) \rceil $ for every graph $G$.
\end{conj}

 Kawarabayashi and the second author~\cite{ks} proved that $oh(G)\ge   \lceil |G|/(2\alpha(G)-1) \rceil $ for every graph $G$. This result is analogous to a classical bound by Duchet and Meyniel~\cite{dm}. In the same paper,  Kawarabayashi and the second author~\cite{ks} further improved the bound for graphs with independence number three,  showing that  $oh(G)\ge   \lceil |G|/4 \rceil $ in this case. 
 One particular interesting case of Hadwiger's Conjecture is  when graphs have independence number two. It has attracted more attention recently (see Section~4 in Seymour's survey~\cite{sem} for more information). As  stated in his survey,  Seymour believes that if Hadwiger's Conjecture is true for graphs $G$ with $\alpha(G)=2$, then it is probably true  in general.  Recall that Hadwiger's Conjecture states that $h(G)\ge \chi(G)$.   Plummer, Stiebitz, and Toft~\cite{pst} proved that  Hadwiger's Conjecture and Duchet-Meyniel Conjecture are equivalent for graphs $G$ with $\alpha(G)\le2$; in the same paper they also proved that  Hadwiger's Conjecture holds for every $H$-free graph $G$ with $\alpha(G)=2$, where $H$ is any graph on four vertices and $\alpha(H)=2$.  Kriesell~\cite{kri} then extended their result and proved that Hadwiger's Conjecture holds for every $H$-free graph $G$ with $\alpha(G)=2$, where $H$ is any graph on five vertices and $\alpha(H)=2$.  Theorem~\ref{t:clique}  is a deep result of Chudnovsky and Seymour~\cite{cs}. 
\begin{thm}[Chudnovsky and Seymour~\cite{cs}]\label{t:clique}
Let $G$ be a graph on $n$ vertices with $\alpha(G)\le2$.   If 
\[
\omega(G)\ge \begin{cases}
   n/4, & \text{if $n$ is even}\\
   (n+3)/4, & \text{if n is odd,}
\end{cases}
\]
then   $h(G) \ge \chi(G)$.
\end{thm}
 \medskip

  Following the ideas of Plummer, Stiebitz and Toft in~\cite{pst}, one can easily prove that  Odd Hadwiger's Conjecture and Conjecture~\ref{c:oddalpha2} are equivalent for graphs $G$ with $\alpha(G)\le2$. We provide a proof here for completeness. 

\begin{thm}\label{oddalpha2} Let $G$ be a graph on $n$ vertices  with $\alpha(G)\le2$. Then 
\[oh(G)\ge \chi(G)  \text{ if and only if  } oh(G)\ge \lceil n/2\rceil.\]
\end{thm}
\begin{proof} Since $\alpha(G)\le 2$, we see that  $\chi(G)\ge \lceil n/2\rceil$, that is, $n\le 2\chi(G)$. It suffices to show that if $oh(G)\ge \lceil n/2\rceil$, then $oh(G)\ge \chi(G)$. Suppose  $oh(G)\ge \lceil n/2\rceil$ but $oh(G)< \chi(G)$. We choose $G$ with $n$ minimum.  
If  $\chi(G \less v) = \chi(G)$ for some  vertex $v \in V (G)$, then  by the minimality of $n$,    $oh(G)\ge oh(G - v)\ge \chi(G\less v)=\chi(G)\ge \lceil n/2\rceil$, a contradiction.   Thus  $G$ is $\chi(G)$-critical.  Suppose  $n = 2\chi(G)$. Then for every vertex $v\in V(G)$, by the minimality of $n$, we have  $oh(G-v)\ge  \chi(G - v)$.  Thus  
\[oh(G)\ge oh(G\less v)\ge \chi(G \less v) \ge \lceil (n-1)/2\rceil =\chi(G),\]
 a contradiction. This proves that $n \leq 2\chi(G) - 1$.  

Suppose next  $\overline{G}$ is disconnected.  Then $G$  is the join of  two vertex-disjoint graphs, say  $G_1$ and $G_2$. By the minimality of $n$, $oh(G_i) \ge  \chi(G_i)$ for each $i\in[2]$. It follows that  \[oh(G)\ge oh(G_1)+oh(G_2)\ge \chi(G_1)+\chi(G_2)=\chi(G),\] 
a contradiction.  Thus $\overline{G}$ is connected.    By a deep result of Gallai~\cite{gal} on the order of $\chi(G)$-critical graphs, we have $n = 2\chi(G) - 1$.    But then $oh(G)<\chi(G)= \lceil n/2\rceil$, contrary to our assumption that $oh(G)\ge \lceil n/2\rceil$.   
\end{proof}

A \dfn{seagull} in a graph is an induced path on three vertices.  It is worth noting that  Theorem~\ref{t:clique} follows from Theorem~\ref{t:seagulls2} on the existence of pairwise disjoint seagulls in graphs $G$ with $\alpha(G)\le2$.

\begin{thm}[Chudnovsky and Seymour~\cite{cs}]\label{t:seagulls2}  Let $G$ be a graph on $n$ vertices with $\alpha(G)\le2$ such that $G$ is $\ceil{\frac n2}$-connected. Let $K$ be a largest clique in $G$.  If 
\[\frac32\ceil{\frac n2}-\frac n2\le |K|\le \ceil{\frac n2},\]
then $G\less K$ contains $\ceil{n/2}-|K|$ pairwise disjoint seagulls. 

\end{thm}

Lemma~\ref{l:highconnected}, a recent result by   Chen and Deng~\cite{cd},  plays a key role in establishing the existence of odd bipartite minors in graphs $G$ with $\alpha(G)\le2$. 

\begin{lem}[Chen and Deng~\cite{cd}]\label{l:highconnected}
Let $G$ be a graph on $n$ vertices with $\alpha(G)\le2$.  If $G$ is not $\ceil{n/2}$-connected, then $oh(G)\ge \ceil{n/2}$. 

\end{lem}

Utilizing Theorem~\ref{oddalpha2},  Theorem~\ref{t:seagulls2} and Lemma~\ref{l:highconnected}, we next extend Theorem~\ref{t:clique}  to odd clique minors, that is, we prove that Odd Hadwiger's Conjecture holds for graphs  $G$ with $\alpha(G)\le2$ and $\omega(G)$ sufficiently large. We include the   proof here as it is very short. 
\begin{thm}\label{t:oclique}
Let $G$ be a graph on $n$ vertices with $\alpha(G)\le2$.   If 
\[
\omega(G)\ge \begin{cases}
   n/4, & \text{if $n$ is even}\\
   (n+3)/4, & \text{if n is odd,}
\end{cases}
\]
then $oh(G) \ge \chi(G)$.
\end{thm}
 
\begin{proof} By Theorem~\ref{oddalpha2}, it suffices to show that $oh(G) \ge \ceil{n/2}$. Suppose $oh(G) <\ceil{n/2}$. By Lemma~\ref{l:highconnected}, $G$ is $\ceil{\frac n2}$-connected. Let $K$ be a largest clique in $G$. It can be easily checked that 
\[\frac32\ceil{\frac n2}-\frac n2\le |K|  \le \ceil{\frac n2}.\]
By Theorem~\ref{t:seagulls2}, we see that $G\less K$ contains $\ceil{n/2}-|K|$ pairwise disjoint seagulls. Such seagulls, together with $K$, yield an odd clique minor of order $\ceil{n/2}$ in $G$, which contradicts to the assumption that $oh(G) <\ceil{n/2}$.
\end{proof}

 B. Thomas and the second author~\cite{st} in 2017  proved that Odd Hadwiger's Conjecture holds for    $\{C_4,   \dots, C_{2\alpha(G)}\}$-free graphs.

\begin{thm}[Song and B. Thomas~\cite{st}]\label{forbiddenholes}
Every   $\{C_4,    \dots, C_{2\alpha(G)}\}$-free graph $G$   satisfies   $oh(G)\ge \chi(G)$.
\end{thm}

Theorem~\ref{forbiddenholes} implies that  Odd Hadwiger's Conjecture holds for $C_4$-free graphs with independence number at most two. We define the   \dfn{kite} to be the graph obtained from $K_4^-$ by adding a new vertex adjacent to a vertex of degree two in $K_4^-$. Our goal is to provide more evidence for Conejcture~\ref{c:oddalpha2} by extending the results of Plummer, Stiebitz and Toft~\cite{pst}  and Kriesell~\cite{kri} to odd minors. In contrast to the construction of clique minors, the primary challenge lies in the fact that a dominating edge is insufficient for constructing the desired odd clique minor, where an edge $xy$ in a graph $G$ is \dfn{dominating} if every vertex in $V(G)\less\{x,y\}$ is adjacent to $x$ or $y$.  The main purpose of our paper is to prove that 
  Odd Hadwiger's Conjecture holds for $H$-free graphs $G$ with $\alpha(G)\le2$, where $H$ is given in Theorem~\ref{main}.    

\begin{thm}\label{main}
Let $G$ be an $H$-free graph  on $n$ vertices with  $\alpha(G)\leq 2$, where    $\alpha(H)\le 2$ and $H$ is an induced subgraph of  
 $K_1 + P_4$,  $K_2+(K_1\cup K_3)$,  $K_1+(K_1\cup K_4)$, $K_7^-$, $K_7$,  or the     kite graph. 
Then  $oh(G)\ge   \chi(G)$.
\end{thm}

We prove Theorem~\ref{main} in Section~\ref{s:main}. Our proof does not rely on Theorem~\ref{t:oclique} in the hope that our method will shed some light on obtaining more evidence for Odd Hadwiger's Conjecture.  Some open problems are listed in Section~\ref{s:remarks}.  We end this section by 
  first list some known results and then prove two lemmas that shall be
applied in the proof of  Theorem~\ref{main}.

\begin{lem}[Song and B. Thomas~\cite{st}]\label{inflation}
Let $G$ be an inflation of an odd cycle $C$, that is, $G$ is  is obtained from $C$ by replacing each vertex  of $C$ by a clique of order at least one and two such cliques are complete to each other if their corresponding vertices in $C$ are adjacent. Then   $oh(G) \geq \chi(G)$.
\end{lem}

For a nonempty clique $K$ in a graph $G$, let $K^*$  be the set of vertices $v \in V (G) \less K$ such that   $v$ is neither complete nor anti-complete to $K$.   We define the \dfn{capacity of $K$}  to be 
$(|G|+|K^*|-|K|)/2$.  

\begin{thm}[Chudnovsky and Seymour]\label{t:seagulls}  Let $G$ be a graph with $\alpha(G)\le2$ and let $\ell\ge0$  be an integer such that
 $G\ne K_1+C_5$  when $\ell=2$. Then $G$ has $\ell$ pairwise disjoint seagulls  if and only if
\begin{enumerate}[(i)]
\item  $|G|\ge 3\ell$,
\item $G$ is $\ell$-connected,
\item every clique of $G$ has capacity at least $\ell$ and
\item  $G$ admits an anti-matching of cardinality $\ell$, that is,  $\overline{G}$ admits a   matching of cardinality $\ell$.
\end{enumerate}
\end{thm}

\begin{lem}\label{L1}
Let $G$ be a graph on $n$ vertices.
If $\delta(G) \geq n-3$, then $ oh(G)\ge \lceil n/ \alpha(G) \rceil$.
\end{lem}

\begin{proof}
Suppose the statement is false, that is, $ oh(G)< \lceil n/ \alpha(G) \rceil$. We choose $G$ with $n$ minimum. Since $\delta(G) \geq n-3$, we see that  $\Delta(\overline{G}) \leq 2$ and $2\le \alpha(G)\le3$. It follows that  $\overline{G}$ consists of disjoint union of   paths and cycles. One can easily check that $oh(G)\ge \omega(G)=\alpha(\overline{G})\geq \lceil n/ 3 \rceil$.   It follows that   $\alpha(G) = 2$.   If $\overline{G}$ contains no odd cycles, then $\overline{G}$ is bipartite and so $oh(G)\ge \omega(G)=\alpha(\overline{G}) \geq \lceil n/ 2 \rceil$, a contradiction. Thus    $\overline{G}$ contains an odd cycle $C$, say,  with vertices $a_1, \ldots, a_{2k+1}$ in order, where $k \geq 2$ (since $\alpha(G)=2$). By the minimality of $n$, we see that  $oh(G \less V(C))\ge  \lceil (n-2k-1)/2\rceil$. Adding the $k$ vertices $a_1,a_3,...,a_{2k-1}$ and the seagull $G[\{a_{2k}, a_2, a_{2k+1}\}]$, we obtain an odd clique minor of order $  \lceil {(n-2k-1)}/{2}\rceil+k+1 \ge \lceil   n/2\rceil$ in $G$, a contradiction.
\end{proof}

\begin{lem}\label{L2}
 Let $G$ be a graph on $n$ vertices with  $\alpha(G) \leq 2$ and $x, y\in V(G)$. If  $xy\notin E(G)$ and $y$ is complete to $N(x)$ in $G$ and $oh(G\backslash\{x,y\})\ge \ceil{n/2}-1$, then $ oh(G)\ge \lceil  n/2 \rceil$.
\end{lem}

\begin{proof}
 Since  $\alpha(G)\le2$ and $y$ is complete to $N(x) $ in $G$, we see that   $y$ is complete to $V(G)\backslash\{x,y\}$  in $G$.  Therefore,  $ oh(G)\ge oh(G\backslash\{x,y\})+1\ge \lceil   n/2\rceil$, as desired.
\end{proof}

\section{Proof of Theorem \ref{main}}\label{s:main}
Let $G$ be an $H$-free graph  on $n$ vertices with  $\alpha(G)\leq 2$. By Theorem~\ref{oddalpha2}, it suffices to prove that  $h(G)\ge \ceil{n/2}$ when $H$ is isomorphic to $K_1 + P_4$,  $K_2+(K_1\cup K_3)$,  $K_1+(K_1\cup K_4)  $, $K_7^-$, $K_7$, or the     kite graph.   Suppose the statement is false. Then $oh(G) < \ceil{n/2}$. We choose $G$ with $n$ minimum.  Then $G$ is connected,  $\omega(G)\le \ceil{n/2}-1$ and $\alpha(G) = 2$.  We next prove several claims. \\

\setcounter{counter}{0}

\noindent {\bf Claim\refstepcounter{counter}\label{degree}  \arabic{counter}.}
$\delta(G)\le n-4$ and $\Delta(G)\le n-3$. 
 \begin{proof}
 By Lemma~\ref{L1}, we have $\delta(G) \leq n - 4$.  Suppose  $\Delta(G)= n-1$.  Let  $x\in V(G)$ with $d_G(x)=n-1$. By the minimality of $n$, $oh(G\less\{x\})\ge \lceil   (n-1)/2\rceil$ and so  $oh(G)\ge oh(G\less\{x\})+1\ge \lceil   n/2\rceil$, a contradiction.  This proves that $\Delta(G)\le n-2$. By Lemma~\ref{L2}, we have   $\Delta(G)\le n-3$.  
 \end{proof}

\noindent {\bf Claim\refstepcounter{counter}\label{nnbr}  \arabic{counter}.} 
For any $v\in V(G)$, $V(G) \less N[v]$  is a clique, and  $2\le n-1-d_G(v)\le \omega(G)\le \ceil{n/2}-1$.
\begin{proof}Since $\alpha(G)=2$, we see that  for any $v\in V(G)$, each pair of vertices in $V(G) \less  N[v]$ are adjacent in $G$. Thus  $V(G) \less N[v]$  is a clique in $G$. By Claim~\ref{degree},  \[2\le n-1-d_G(v)=|V(G) \less N[v]|\le \omega(G)\le \ceil{n/2}-1,\]
  as desired. \end{proof}
\noindent {\bf Claim\refstepcounter{counter}\label{K}  \arabic{counter}.} 
For each clique $K$ in $G$, $K$ is not complete to $V(G)\less K$ in $G$, $G\less K$ is connected and  $\alpha(G\less K)=2$. 
\begin{proof}
 Let $K$ be a clique in $G$. Suppose $K$ is  complete to $V(G)\less K$ in $G$. By the minimality of $n$, $G\less K$ contains an odd clique minor of order $\ceil{(n-|K|)/2}$. This, together with $K$, yields an odd clique minor of order $\ceil{n/2}$ in $G$, a contradiction. 
 Suppose next  $G\less K$ is disconnected. 
 Then $V(G)\less K$ is the disjoint  union of two cliques, say, $A$ and $B$. Then each vertex in $K$ is either complete to $A$ or complete to $B$. Define
\[K':=\{u\in V(G)\less K\mid u \text{ is not complete to } B\} \text{ and } K'':= K\less K'.\]
Then $K'\cup A$ and $K''\cup B$ are disjoint cliques in $G$. Thus $\omega(G)\ge \ceil{n/2}$, a contradiction. This proves that $G\less K$ is connected. Since  $\omega(G)<\ceil{n/2}$, we see that   $\alpha(G\less K)=2$. 
\end{proof}

\noindent {\bf Claim\refstepcounter{counter}\label{nodd}  \arabic{counter}.}   
$n$ is odd. 
 
 \begin{proof}Suppose $n$ is even. Let  $x\in V(G)$. By the minimality of $n$,    $oh(G)\ge oh(G\less\{x\}) \ge \ceil{(n-1)/2}=\ceil{n/2}$, a contradiction. 
 \end{proof}

\noindent {\bf Claim\refstepcounter{counter}\label{ell}  \arabic{counter}.} 
If $\omega(G)= \ceil{n/2}-\ell$ for some integer $\ell\ge1$, then  $n\le 4\ell-1$. Moreover,  $\ell\ge \omega$ and $2\omega(G)\le \ceil{n/2}$. 

\begin{proof} Suppose $n\ge 4\ell+1$. 
Since $\omega(G)= \ceil{n/2}-\ell$, we see that $n=2\omega(G)+2\ell-1$. 
Let $D$ be a maximum clique in $G$ and $G' := G \backslash D$. Then $\omega(G') \leq \omega(G)$. We next show that $G'$ contains $\ell  $ pairwise disjoint seagulls. Note  that $|G'| =n-\omega(G)=\lfloor n/2\rfloor+\ell\geq 3\ell$ because $n\ge 4\ell+1$. Moreover, $G'$ admits an $\ell$ anti-matching, else $\omega(G')\ge n-\omega(G)-2(\ell-1) = \omega(G) +1$, a contradiction. For each clique $K$ in $G'$, we see that $|K| \leq \omega(G)$. Recall that  $K^*$ denotes the set of vertices $v$ in $G'\less K$ such  that  $v$ is neither complete nor anti-complete to $K$. We claim that  $|K|-|K^*|\le \omega(G)-1$. This is trivially true   when $|K| \le  \omega(G)-1$. Suppose   $|K| = \omega(G)$. Then $|G' \backslash K| = n-2\omega(G)=2\ell-1\ge1$ and no vertex in $V(G')\less K$ is complete to $K$ in $G$. By Claim~\ref{K}, $G'$ is connected and so $K$ is not anti-complete to  $V(G')\less K$.   Thus $G' \backslash K$ has at least one vertex that 
is neither complete nor anti-complete to $K$ in $G'$, and so  $|K|-|K^*|\le \omega(G)-1$. This proves that  $|K|-|K^*|\le \omega(G)-1$, as claimed.  It follows that  $K$ has capacity at least $(n-2\omega(G)+1)/2=\ell$ in $G'$. Finally, suppose $G'$ is not $\ell$-connected. Let $S$ be a minimum vertex cut in $G'$. Then $|S| \leq \ell-1$ and $G' \backslash S$ has exactly two components, say $G_1$ and $G_2$. Note that $V (G_1)$ and $V (G_2)$ are disjoint cliques; moreover, $V (G_1)$ is anti-complete to $V (G_2)$. We can then partition $D$ into $D'$ and $D''$ such that $V (G_1) \cup D'$
and $V (G_2) \cup  D''$ are disjoint cliques. But then $\omega(G)\geq  \max\{|V (G_1) \cup D'|, |V (G_2) \cup  D''|\} \geq \ceil{(n-\ell+1)/2}>\omega(G)$, a
contradiction. This proves that $G'$ is $\ell$-connected. By Theorem~\ref{t:seagulls} applied to $G'$ and $\ell$, we see
that $G'$ has $\ell$ pairwise disjoint seagulls. Such $\ell$ seagulls, together with the clique $D$, yield an
odd clique minor of order $\lceil n/2\rceil$ in $G$, a contradiction. This proves that $n\le 4\ell-1$. Since $n=2\omega(G)+2\ell-1$, we see that $2\omega(G)+2\ell-1=n\le 4\ell-1$. Thus   $\ell\ge \omega$ and so $2\omega(G)\le \ceil{n/2}$. 
\end{proof}

\noindent {\bf Claim\refstepcounter{counter}\label{clique}  \arabic{counter}.} 
$\omega(G)\ge7$ and so $\omega(G) \leq \ceil{n/2} -7$ and $n\ge 27$. Moreover, $\omega(G) = \ceil{n/2} -7$ if and only if  $n=27$. 

\begin{proof} Suppose $\omega(G)\le 6$.   Then $n\le R(3, 7)-1=22$. Let $\omega(G):=\ceil{n/2}-\ell$ for some integer $\ell$. Then $n=2\omega(G)+2\ell-1$ and  $\ell\ge1$ because $\omega(G)\le oh(G)<\ceil{n/2}$.  By Claim~\ref{ell},  $\ell\ge \omega$. It follows that $n=2\omega(G)+2\ell-1\ge 4\omega(G)-1$. Note that  $R(3, \omega(G)+1)\le 4\omega(G)-1$ when $\omega(G)\le 6$.   Thus  $G$ contains a clique of order $\omega(G)+1$, which is impossible.  This proves that $\omega(G)\ge 7$. Thus   $7\le \omega(G) \le\ceil{n/2}-1$ and so  $n\ge 15$.  Let $\omega(G) := \ceil{n/2} -\ell$. Then $\ell\ge1$. Since $n\ge 15$ and $R(3,5)=14$, by Claim~\ref{ell}, $\ell\ge\omega(G)\ge5$.  Thus $7\le \omega(G) \leq \ceil{n/2} -5$  and so $n\ge23$. Since $R(3,7)=23$, by Claim~\ref{ell} again, $\ell\ge7$.  Thus $7\le \omega(G) \leq \ceil{n/2}-7$  and so $n\ge27$; in addition if $n=27$, then $  \omega(G) = \ceil{n/2} -7$. By Claim~\ref{ell}, if $\omega(G) = \ceil{n/2} -7$, then  $n=27$. 
\end{proof}

\noindent {\bf Claim\refstepcounter{counter}\label{nbr}  \arabic{counter}.} 
 For any $v\in V(G)$, $N(v)$ is not a clique. Furthermore, $\alpha(G[N(v)]) = 2$ and $G[N(v)]$ is connected.
\begin{proof}
Let $v\in V(G)$.  By Claim~\ref{nnbr}, $V(G)\less N[v]$ is a clique. Thus $N(v)$ is not a clique and  
 $\alpha(G[N(v)]) = 2$  by Claim~\ref{K}. Suppose  $G[N(v)]$ is disconnected. Then $N(v)$ is the disjoint union of two   cliques, say $A_1$ and $A_2$. For each $u\in V(G)\less N[v]$, $u$ is complete to $A_1$ or $A_2$ in $G$. Define
\[B_1:=\{u\in V(G)\less N[v]\mid u \text{ is not complete to } A_2\} \text{ and } B_2:= V(G)\less (N[v]\cup B_1).\]
Then $A_1\cup B_1$    and $A_2\cup B_2$ are disjoint cliques in $G$. Then   $ \omega(G)=|A_1\cup B_1|= |A_2\cup B_2|=\ceil{(n-1)/2}$. Then $B_1\ne \es$ and $B_2\ne \es$. Let $y\in A_1$ and $z\in B_1$. Then the seagull $G[\{x, y, z\}]$, together with $A_2\cup B_2$, yields an odd clique minor of order  $ \ceil{n/2}$ in $G$, a contradiction.  
\end{proof}

Throughout  the remaining of the proof, let $x, y$ be two non-adjacent vertices in $G$.  
Since $\alpha(G)=2$, we see that every vertex in $V(G)\less \{x,y\}$ is adjacent to either $x$ or $y$. Let $A: =N(x)\backslash N(y) $, $B: =N(x)\cap N(y)$ and $C: =N(y)\backslash N(x) $.  Then both $A\cup\{x\}$ and $C\cup\{y\}$ are cliques in $G$ by Claim~\ref{nnbr}. \\

\noindent {\bf Claim\refstepcounter{counter}\label{ABC}  \arabic{counter}.}  
  $|B| \geq 13$ and  neither $A$ nor $C$ is anti-complete to $B$. Moreover, if  $|A|\le |C|$, then   $C$ is   not complete   to $B$.  
\begin{proof}
By Claim~\ref{degree},      $|A|\ge 1$ and  $|C|\ge1$.    By Claim~\ref{clique},  $\max\{|A \cup \{ x \}|, |C \cup \{ y \}|\}  \le \ceil{n/2}-7$.  Thus  $|B|=n-(|A \cup \{ x \}|+|C \cup \{ y \}|) \ge 13 $. By Claim~\ref{nbr} applied to $G[N(x)]$ and $G[N(y)]$, we see that   neither $A$ nor $C$ is anti-complete to $B$.    Suppose next $|A|\le |C|$ and  $ C$ is   complete to $B$. Then $C\cup \{y\}$, together with an odd clique minor of order   $\ceil{ |B|/2}$ in $G[B]$, yields an odd clique minor of order   $  \ceil{n/2}$ in $G$, a contradiction.   This proves that  if $|A|\le |C|$, then   $C$ is not complete   to $B$. 
  \end{proof}

 \noindent {\bf Claim\refstepcounter{counter}\label{maxdeg}  \arabic{counter}.}  
 $\ceil{n/2}+5\le \delta(G)\le n-5$.  
 
 \begin{proof}
 By Claim~\ref{clique} and Claim~\ref{nnbr}, we see that $\delta(G)\ge n-1-\omega(G)\ge \ceil{n/2}+5$. Suppose $\delta(G)= n-4$.  Then $\Delta(G)= n-3$, else $G$ is $(n-4)$-regular on $n$ vertices, which is impossible because $n$ is odd. We may assume that $d_G(y)= \Delta(G)$.  Then $ |A|=1$. Let $A:=\{z\}$.   If  there exists a vertex $u\in N(y)$ such that $G[\{x,z, u\}]$ is a seagull, then  $G[\{x,z, u\}]$, together with $y$ and an odd  clique minor of order $\ceil{|N(y)\less \{u\}|/2}=\ceil{n/2}-2$ in $G[N(y)\less \{u\}]$, yields an odd clique minor of order $\ceil{n/2}$ in $G$,  a contradiction. Thus  $z$ is complete to $B$.     Note that    $1\le |C|\le2$, else $d_G(x)\le n-5$. By the minimality of $n$, $G[B]$ has an odd clique minor of order $\ceil{(n-3-|C|)/2}$.  This, together with $x$ and $z$, yields an odd clique minor of order $\ceil{(n+1-|C|)/2}$ in $G$. It follows that $|C|\ge2$ and so $|C|=2$.    
  By Claim~\ref{ABC}, we see that   $C$ is  not complete   to $B$.   Thus  there exists a vertex $u\in B$ such that $G[\{u,y,v\}]$ is a seagull, where $v\in C$. But then $G[\{u,y,v\}]$, $x,z$, together with an odd clique minor of order $\ceil{|B\less\{u\}|/2}=\ceil{(n-6)/2}$ in $G[B\less\{u\}]$, yield an odd clique minor of order $\ceil{n/2}$ in $G$, a contradiction. This proves that $\delta(G)\le n-5$, as desired. 
  \end{proof}

 Recall that $G$ is $H$-free.   
 By Claim~\ref{clique}, $H\ne K_7$. 
 We  consider  the following five cases when $H$ is isomorphic to   $K_1+P_4$,  $ K_2+(K_1\cup K_3)$, $K_1+(K_1\cup K_4)$, $K_7^-$,   or the kite graph.       \\

\noindent {\bf Case 1.}  $H = K_1 + P_4$.\\

\noindent In this case we choose $x$ with $d_G(x) = \delta(G)$. Then $|C| \geq |A|$. By Claim~\ref{ABC}, we see that   $C$ is  not complete  to $B$. Let $z \in B$ and $w \in C$ such that $zw \notin E(G)$. If there exists $b \in B $ with $b\ne z$ such that $bz, bw \in E(G)$, then $G[\{b, x, z, y, w\}] = K_1 + P_4$, a contradiction. Thus we can partition $B$ into $B_1$ and $B_2$ such that 
\[B_1:= \{b\in B\mid bw \notin E(G)\} \text{ and } B_2 := \{ b\in B\mid bz \notin E(G)\}.\]
 Then $z\in B_1$. If there exist $b_1 \in B_1$ and $b_2 \in B_2$ such that $b_1b_2 \in E(G)$, then $b_1\ne z$ and $G[\{y, z, b_1, b_2, w\}] = K_1 + P_4$, a contradiction. Thus $B_1$ is anti-complete to $B_2\cup \{w\}$, and so  $B_1$ and $B_2\cup \{w\}$ are two cliques in $G$.  We claim that $C$ is complete to $B_2$. Suppose there exist $c\in C $ and $b_2 \in B_2$ such that $cb_2 \notin E(G)$. Then $c\ne w$ and $cz \in E(G)$ because $b_2z\notin E(G)$. But then $G[\{y, z, c, w, b_2\}] = K_1 + P_4$, a contradiction. This proves that  $C$ is complete to $B_2$, as claimed.  Note that $B_1$ and $C \cup B_2$ are two disjoint cliques.  If there exist $c \in C $ and $b_1, b_1' \in B_1 $ such that $cb_1 \notin E(G)$ and $cb_1' \in E(G)$, then $c\ne w$ and $G[\{y, b, b', c, w\}] = K_1 + P_4$, a contradiction. Thus each vertex in $C$ is either complete or anti-complete to $B_1$.  We  partition $C$ into $C_1$ and $C_2$ such that 
 \[C_1:= \{c\in C \mid c \text{ is anti-complete to } B_1 \}  \text{ and } C_2:= \{c\in C  \mid c \text{ is complete to } B_1\}.\] 
 Then $w\in C_1$ and  $B_1$ is anti-complete to $B_2 \cup C_1$.   By Claim~\ref{nbr} applied to $G[N(y)]$, we have $C_2\ne\es$.  Finally we partition $A$ into $A_1$ and $A_2$ such that
\[A_1:= \{a\in A \mid a \text{ is  not complete to } B_2\cup C_1 \}  \text{ and } A_2:= \{a\in A \mid a \text{ is  complete to } B_2\cup C_1\}.\] 
Then $A_1$ is complete to $B_1\cup A_2$ and $A_2$ is complete to $B_2\cup C_1$. Note that  $A_1\ne \es$ and $A_2 \neq \emptyset$, else $\omega(G) \geq \lceil(n-1)/2\rceil$, contrary to Claim~\ref{clique}. Then $B_2 =\emptyset$, else say  $b_2 \in B_2$, then  $G[\{b_2, x, a_2, c_1, y\}] = K_1 + P_4$, where  $a_2 \in A_2$ and $c_1 \in C_1$, a contradiction. Thus $A_1$ is not complete to $C_1$. By Claim~\ref{nbr} applied to $G[N(x)]$, we see that $B=B_1$ is not complete to $A_2$.   Let $b \in B $ and $a_2 \in A_2$ such that $ba_2 \notin E(G)$. Then $a_2$ is anti-complete to $B$, else say $a_2b'\in E(G)$ for some $b'\in B$, then   $G[\{b', y, b, x, a_2\}] = K_1 + P_4$, a contradiction. Suppose there exists $a_2'\in A_2$ such that $a_2'$ is adjacent to a vertex $b'$ in  $B$. Then $G[\{a_2', b', a_1, a_2, c_1\}] = K_1 + P_4$, where $a_1\in A_1$ and $c_1\in C_1$ such that $a_1c_1\notin E(G)$, a contradiction. This proves that $B$ is anti-complete to $A_2$. Then   $A_1$ is  anti-complete to $C_1$, else say $a_1c_1\in E(G)$ for some  $a_1 \in A_1 $ and $c_1 \in C_1$, then $G[\{a_1,  b, x, a_2, c_1\}] = K_1 + P_4$, a contradiction.  Finally we prove that $A$ is anti-complete to $C_2$. Suppose  there exist $ a\in A$ and $c_2 \in C_2$ such that $ac_2 \in E(G)$.   Then $G[\{b, x, a, c_2, y\}] = K_1 + P_4$ when $a\in A_1$; $G[\{c_2, a, w, y,z\}] = K_1 + P_4$ when $a\in A_2$, a contradiction.  This proves that $A$ is anti-complete to $C_2$. \medskip

We now partition $V(G)$ into $X_1, \ldots, X_5$, where 
\[X_1:=B=B_1, X_2:=C_2\cup\{y\}, X_3:=C_1, X_4:=A_2 \text{ and } X_5:=A_1\cup\{x\}.\] Then   $X_1$ is anti-complete to $X_3\cup X_4$, $X_2$ is anti-complete to $X_4\cup X_5$, $X_3$ is anti-complete to $X_5$,    $X_i$ is complete to $X_{i+1}$ for each $i\in[4]$, and $X_1$ is complete to $X_5$. It follows that  $G$ is an inflation of $C_5$. By Lemma~\ref{inflation}, $oh(G)\ge\chi(G)\ge \ceil{n/2}$, a contradiction. \bigskip

\noindent {\bf Case 2.} $H = K_2 + (K_1\cup K_3)$.\medskip

\noindent In this case we choose $x$ with $d_G(x)=\delta(G)$. Subject to the choice of $x$, we choose $y$ so that $d_G(y)$ is minimum among all the vertices in $V(G)\less N[x]$.   Then $|A|\le |C|$ and $|C|\ge3$. Let   $C:=\{c_1, \ldots, c_{|C|}\}  $.  For $i, j\in[|C|]$ with $i\ne j$, if $G[B\cap N(c_i)\cap N(c_j)]$ contains an edge, say $uv$,  then   $G[\{u, v, x, c_1,c_2, y\}  ]=H$, a contradiction.  Thus $B\cap N(c_i)\cap N(c_j)$ is an independent set of order at most two, and so $e_G(\{c_i, c_j\}, B)\le |B|+2$.  Since $|C|\ge 3$, we may assume that $e_G(c_1,B)\le \ceil{(|B|+4)/3}$. By the choice of $y$, we see that 
\[|B|+|C|=d_G(y)\le d_G(c_1)\le e_G(c_1,B)+|C|+|A|\le \ceil{(|B|+4)/3}+|C|+|A|,\]
which implies that $|B|\le \ceil{(|B|+4)/3}+|A|$ and so $ |A|\ge7$ because $|B|\ge13$.
 Let \[B_1:= \{b \in B\mid b \text{ is anti-complete to } C\} \text { and } B_2:  = B\less B_2.\]
 Then $B_1$ is a clique. By Claim~\ref{nbr} applied to $G[N(y)]$,  $|B_2|\ge1$. Note that every vertex $b \in B_2$ is adjacent to  at least $|C|-2$  vertices  in $C$, else $G[\{b, y\} \cup C]$ is not  $H$-free.    Recall that for each pair of vertices $c_i, c_{i+1}$ in  $C$, $e_G(\{c_{2i-1}, c_{2i}\}, B_2)\le |B_2|+2$ for each $ i\in[ \lfloor{|C|/2}\rfloor]$. Then 
 \[(|C|-2)|B_2|\le e_G(B_2, C)\le \left\lfloor{\frac{|C|}2}\right\rfloor (|B_2|+2) +\left(\left\lceil{\frac{|C|}2}\right\rceil-\left\lfloor{\frac{|C|}2}\right\rfloor\right)(|B|+2)\le\left\lceil{\frac{|C|}2}\right\rceil(|B_2|+2).\]
  It follows that $|B_2|\le 8$ because  $|C|\ge|A|\ge7$. Note that  $B_1$ and $C$ are two disjoint cliques and $B_1$ is anti-complete to $C$. We  now partition $A$ into $A'$ and $A''$ such that $A'\cup B_1$ and $A''\cup C$ are two disjoint cliques in $G$. But then   $\omega(G)\ge \max\{|A'\cup B_1|, |A''\cup C|\}\ge \ceil{(n-|B_2|-2)/2}\ge \ceil{n/2}-5$, contrary to Claim~\ref{clique}. \\

\noindent {\bf Case 3.} $H =K_1 + (K_1\cup K_4)$.\medskip

\noindent   In this case we choose $x$ with $d_G(x)=\delta(G)\le n-5$.  Then $|C|\ge3$.  Let $b\in B$.  If  $b$ is   not adjacent to at least four vertices in $B\cup C$, then $G[N[y]]$ is not  $H$-free. Thus $b$ is not adjacent to at most three vertices in   $C \cup B$.   Moreover, if 
  $b$ is adjacent to  three vertices in $C$, say $c_1, c_2, c_3$, then $G[\{b, x, y, c_1, c_2, c_3 \}]=H$, a contradiction. Thus $b$ is adjacent to at most two vertices in $C$.   Thus   $  |C|\le 5$  and 
  \[2(\delta(G)-|A|) =2|B|\ge e_G(B,   C)\ge |C|(\delta(G)-|C|-|A|).\] 
It follows that $(|C|-2)\delta(G)\le  (|C|-2)|A|+|C|^2$, and so 
\[ |B|+|A|=\delta(G) \le |A|+|C|+ \left\lfloor\frac{2|C|}{|C|-2}\right\rfloor.\]
 Since  $3\le |C|\le 5$, we see that $|B\le |C|+ \left\lfloor\frac{2|C|}{|C|-2}\right\rfloor\le 9$, contrary to Claim~\ref{ABC}. \\

 \noindent {\bf Case 4.} $H = K_{7}^{-}$.\medskip

\noindent  Since $R(3,5) = 14$ and $G[B]$ is $K_5$-free, we see that $|B|\le 13$. By Claim~\ref{ABC}, $|B|=13$. 
Then $\omega(G)\ge \max\{|A\cup\{x\}|, |C\cup\{y\}| \}\ge \ceil{(n-|B|)/2}= \ceil{n/2}-7$. By Claim~\ref{clique}, we see that  $\omega(G)= \max\{|A\cup\{x\}|, |C\cup\{y\}| \}=\ceil{n/2}-7$ and $n=27$.  Then $|A|=|C|=6$ and so $d_G(x)=d_G(y)=19$. By the arbitrary choice of $x$ and $y$, we see that $G$ is $19$-regular on $27$ vertices, which is impossible. \\

\noindent {\bf Case 5.} $H$ is the kite graph.\medskip

\noindent In this case we assume that $|A|\le |C|$. If there exists a vertex $b \in B$ such that $b$ is neither complete nor anti-complete to $C$, then  $G[\{x, b, y\} \cup C]$  is not kite-free. Thus each vertex in $B$ is either complete or anti-complete to $C$ in $G$. Let 
\[B_1: = \{b\in B \mid b \text{ is anti-complete to } C\} \text { and  } B_2:= \{b\in B \mid b \text{ is  complete to } C\}.\] 
Then $B=B_1\cup B_2$. If $|B_1| \geq 2$, then $G[\{x, y, c\} \cup B_1]$  is not kite-free  for any $c \in C$, a contradiction. Thus  $|B_1| \le 1$. Recall  that $|A|\le |C|$. 
 By the minimality of $n$, we see that  $oh(G[B_2]) \geq \ceil{|B_2|/ 2} = \ceil{(|B|-|B_1|)/2}$. This,  together with $C\cup\{y\}$, implies that  
 \begin{align*}
 \ceil{n/2}>oh(G) &\geq oh(G[B_2]) + |C \cup \{y\}| \\
 &\ge \ceil{(|B|-|B_1|)/2}+|C\cup\{y\}|\\
 &\ge \begin{cases}
 \ceil{|B|/2}+\ceil{(n-|B|)/2} &\text{ if } B_1=\es\\
\ceil{(|B|-1)/2}+\ceil{(n-|B|+1)/2} &\text{ if } |A|\le |C|-1.
\end{cases}
\end{align*}
  It follows that $|B_1|=1$ and $|A|=|C|$. By the arbitrary choice of $x$ and $y$, we see that   $G$ is regular. Let $u\in B_1$. Note that $uy\in E(G)$. Then there exists a vertex $v\in A\cup B_2$ such that $uv\notin E(G)$. But then the seagull $G[\{u, x, v\}]$, together with the odd clique minor of order   $  \ceil{|B|/2}-1$ in $G[B\less  \{u,v\}]$ and $C\cup\{y\}$,  yields an odd clique minor of order $\ceil{n/2}$ in $G$, a contradiction.\bigskip
 
%
%

This completes the proof of Theorem~\ref{main}. 

\section{Concluding Remarks}\label{s:remarks}

Let $2K_2$ denotes the complement of the cycle $C_4$. Micu~\cite{micu} in 2005 proved that   Hadwiger's Conjecture holds for $2K_2$-free graphs.  Plummer, Stiebits and Toft~\cite{pst}  proved that Hadwiger's Conjecture  holds for $C_5$-free graphs $G$ with $\alpha(G)\le2$.  Both results have very short proofs.  
It remains an  intriguing open question whether Odd  Hadwiger's Conjecture also holds for $2K_2$-free graphs and   $C_5$-free graphs  $G$ with $\alpha(G)\le2$. \medskip

Let $W_5:=K_1+C_5$. Relying heavily on Theorem~\ref{t:clique}, Bosse~\cite{bo} proved that Hadwiger's Conjecture   holds for $W_5$-free graphs $G$ with $\alpha(G)\le2$. It would be interesting to explore whether Theorem~\ref{t:oclique} can be used to establish  Odd Hadwiger's Conjecture for the same class of graphs.  \medskip

Finally, Steiner~\cite{ste} proved that Odd Hadwiger's Conjecture holds for line-graphs of simple graphs. A natural direction for further research is to investigate whether Steiner's result can be extended to quasi-line graphs,  where a graph $G$ is \dfn{quasi-line} if for each $v\in V(G)$, $N(v)$ is covered by two cliques. A similar question can be posed for each of the results surveyed in~\cite{cv}.

\end{document}